\documentclass[12pt]{article}
\usepackage{amsopn}
\usepackage{bbold}
\usepackage{cite}
\usepackage{amsmath}
\usepackage{amscd}
\usepackage{amsthm}
\usepackage{ mathrsfs }
\newtheorem*{remark}{Remark}
\usepackage[parfill]{parskip}
\usepackage{caption}
\usepackage{amssymb}

\usepackage[all]{xy}

\newcommand{\play}{\displaystyle}

\newcommand{\rt}{\rightarrow}

\DeclareMathOperator{\sgn}{sgn}

\DeclareMathOperator{\End}{End}

\DeclareMathOperator{\im}{Im}

\DeclareMathOperator{\Gal}{Gal}

\newcommand{\ZZ}{\mathbb{Z}}

\newcommand{\QQ}{\mathbb{Q}}

\newcommand{\CC}{\mathbb{C}}

\newcommand{\triv}{\mathbb{1}}

\newtheorem{theorem}{Theorem}[section]
\newtheorem{lemma}[theorem]{Lemma}
\newtheorem{proposition}[theorem]{Proposition}
\newtheorem{corollary}[theorem]{Corollary}
\newtheorem{conjecture}[theorem]{Conjecture}
\newtheorem{definition}[theorem]{Definition}
\title{Identities between Hecke Eigenforms}
\author{D. Bao}

\begin{document}

\maketitle

\begin{abstract}
In this paper, we study solutions to $h=af^2+bfg+g^2$, where $f,g,h$ are Hecke newforms with respect to $\Gamma_1(N)$ of weight $k>2$ and $a,b\neq 0$. We show that the number of solutions is finite for all $N$.  Assuming Maeda's conjecture, we prove that the Petersson inner product $\langle f^2,g\rangle$ is nonzero, where $f$ and $g$ are any nonzero cusp eigenforms for $SL_2(\ZZ)$ of weight $k$ and $2k$, respectively. As a corollary, we obtain that, assuming Maeda's conjecture, identities between cusp eigenforms for $SL_2(\ZZ)$ of the form $X^2+\sum_{i=1}^n \alpha_iY_i=0$ all are forced by dimension considerations. We also give a proof using polynomial identities between eigenforms that the $j$-function is algebraic on zeros of Eisenstein series of weight $12k$.

\end{abstract}

\section{Introduction}
Fourier coefficients of Hecke eigenforms often encode important arithmetic information. Given an identity between eigenforms, one obtains  nontrivial relations between their Fourier coefficients and may in further obtain solutions  to certain related problems in number theory. 
For instance, let $\tau(n)$ be the $n$th Fourier coefficient of the weight 12 cusp form $\Delta$ for $SL_2(\ZZ)$ given by
$$\Delta=q\prod_{n=1}^{\infty}(1-q^n)^{24}=\sum_{n=1}^{\infty}\tau(n)q^n$$
and define the weight 11 divisor sum function $\sigma_{11}(n)=\sum_{d|n}d^{11}$. Then
 the Ramanujan congruence  $$\tau(n)\equiv \sigma_{11}(n) \mod 691$$

can be deduced easily from the identity $$E_{12}-E_6^2=\frac{1008\times 756}{691}\Delta,$$

where $E_6$ (respectively $E_{12}$) is the Eisenstein series of weight 6 (respectively 12) for $SL_2(\ZZ)$. 

Specific polynomial identities between Hecke eigenforms have been studied by many authors.
Duke\cite{Duke} and Ghate\cite{ghate2000monomial} independently investigated identities of the form $h=fg$ between eigenforms with respect to the full modular group $\Gamma=SL_2(\ZZ)$ and proved that there are only 16 such identities. 
  Emmons\cite{MR2180510} extended the  search to $\Gamma_0(p)$ and found 8 new cases. Later Johnson\cite{johnson2013hecke} studied the above equation for Hecke eigenforms with respect to $\Gamma_1(N)$ and obtained a complete list of $61$ eigenform product identities. 
 J. Beyerl and K. James  and H. Xue\cite{MR3119178} studied the problem of when an eigenform for $SL_2(\ZZ)$ is divisible by another eigenform and proved that this can only occur in very special cases.     
Recently, Richey and Shutty \cite{richey2013polynomial} studied  polynomial identities and showed that  for a fixed polynomial  (excluding trivial ones), there are only finitely many decompositions of normalized Hecke eigenforms for $SL_2(\ZZ)$ described by a given polynomial.  

Since the product of two eigenforms is rarely an eigenform, in this paper we loosen the  constraint and study solutions to the equation $h=P(f,g)$, where $h,f,g$ are Hecke eigenforms and $P$ is a general degree two polynomial. The ring of modular forms is graded by weight, therefore $P$ is necessarily homogeneous, i.e., $P(f,g)=af^2+bfg+cg^2$. With proper normalization, we can assume $c=1$.

The first main result of this paper is the following:

\vspace{.1in}
 \begin{theorem} \label{fin}
	For all $N\in\ZZ_+$ and  $a,b\in\mathbb{C}^{\times}$, there are at most finitely triples $(f,g,h)$ of newforms with respect to $\Gamma_1(N)$ of weight $k>2$ satisfying the equation  
\begin{equation}\label{polyid}
h=af^2+bfg+g^2.
\end{equation}	
\end{theorem}

Theorem \ref{fin} is obtained by estimating the Fourier coefficients of cusp eigenforms and its proof is presented in Section 3. In the same section, as further motivation for studying identities between eigenforms, we also give an elementary proof via identities that: 

\vspace{.1in}
\begin{theorem}	
Let $E_{12k}(z)$ be the Eisenstein series of weight $12k$, and	$$A:=\{\rho\in \mathfrak{F}|E_{12k}(\rho)=0 \text{ for some }  k \in\ZZ_{+} \}$$
	be the set of zeros for all Eisenstein series of weight $12k$ with $k\geq 1$ in the fundamental domain $\mathfrak{F}$ for $SL_2(\ZZ)$. If $j$ is the modular $j$-function and $\rho\in A$, then $j(\rho)$ is an algebraic number. 
 
\end{theorem}

 For a stronger result, which says that $j(\rho)$ is algebraic for any zero $\rho$ of a meromorphic modular form $f=\sum_{n=h}^{\infty}a_nq^n$ for $SL_2(\ZZ)$ with $a_n(h)=1$ for which all coefficients lie in a number field, see Corollary 2 of the paper by 
 Bruinier, Kohnen and Ono \cite{bruinier2004arithmetic}.

Fourier coefficients of a normalized Hecke eigenform are all algebraic integers. As for Galois symmetry between Fourier coefficients of normalized eigenforms in the space $\mathcal{S}_k$ of cusp forms for $SL_2(\ZZ)$ of weight $k\geq 12$ , Maeda (see \cite{Maeda} Conjecture 1.2 ) conjectured that: 

\vspace{.1in}
\begin{conjecture}
	The Hecke algebra over $\QQ$ of $\mathcal{S}_k$ is
	simple (that is, a single number field) whose Galois closure over $\QQ$ has Galois
	group isomorphic to the symmetric group $\mathfrak{S}_n$, where $n = \dim \mathcal{S}_k$.
\end{conjecture}

Maeda's conjecture can be used to study polynomial identities between Hecke eigenforms. See J.B. Conrey and D.W. Farmer\cite{Farmer} and  J. Beyerl and K. James  and H. Xue\cite{MR3119178} for some previous work. In this paper, we prove the following:

\vspace{.1in}
\begin{theorem}\label{main1}
	
	Assuming Maeda's conjecture for $\mathcal{S}_k$ and $\mathcal{S}_{2k}$, if $f\in \mathcal{S}_k$, $g\in \mathcal{S}_{2k}$ are any nonzero cusp eigenforms, then the Petersson inner product $\langle f^2,g\rangle $
	is nonzero.
\end{theorem}

Theorem \ref{main1}  says that, assuming Maeda's conjecture, the square of a cusp eigenform for $SL_2(\ZZ)$ is "unbiased", i.e., it does not lie in the subspace spanned by a proper subset of an eigenbasis. The proof of Theorem \ref{main1} is presented in Section 4. 



\section{Preliminaries and Convention} 
A standard reference for this section is \cite{diamond2006first}.
Let $f$ be a holomorphic function on the upper half plane $\mathcal{H}=\{z\in\CC| \im(z)>0\}$. Let $\gamma=\left(\begin{smallmatrix}a& b\\
 c & d\end{smallmatrix}\right)\in  SL_2(\ZZ)$. Then define the slash operator $[\gamma]_k$ of weight $k$ as $$f|_k{[\gamma]}(z):=(cz+d)^{-k}f(\gamma z).$$ If $\Gamma$ is a subgroup of $ SL_2(\ZZ)$, we say that $f$ is a modular form of weight $k$ with respect to $\Gamma$ if $f|_k[\gamma](z)=f(z)$ for all $z\in\mathcal{H}$ and all $\gamma\in\Gamma$. 
Define

$$\Gamma_1(N):=\left\{\begin{pmatrix}
a& b\\
 c & d \end{pmatrix}\in SL_2(\ZZ)\bigg| \begin{pmatrix}
a& b\\
 c & d \end{pmatrix}\equiv \begin{pmatrix}
1& * \\
 0 & 1 \end{pmatrix} \mod N \right \},$$

$$\Gamma_0(N):=\left\{\begin{pmatrix}
a& b\\
c & d \end{pmatrix}\in SL_2(\ZZ)\bigg|c\equiv 0\mod N \right\}.$$

Denote by $\mathcal{M}_k(\Gamma_1(N))$ the vector space of weight $k$ modular forms with respect to $\Gamma_1(N)$. 

Let $\chi$ be a Dirichlet character mod $N$, i.e., a homomorphism $\chi: (\ZZ/N\ZZ)^\times\rt \CC^\times$. By convention one extends the definition of $\chi$ to $\ZZ$ by defining $\chi (m)=0$ for $\gcd(m,N)> 1$. In particular the trivial character modulo $N$ extends to the function 

$$\mathbb{1}_N(n)=\begin{cases}1 & \text{if } \gcd(n,N)=1,\\
0 &\text{if } \gcd(n,N)>1.\end{cases}$$

Take $n\in (\ZZ/N\ZZ)^\times$, and define the diamond operator 
$$\langle n\rangle: \mathcal{M}_k(\Gamma_1(N))\rt \mathcal{M}_k(\Gamma_1(N))$$

 as $\langle n\rangle f:=f|_k[\gamma]$ for any $\gamma=\left(\begin{smallmatrix}a&b\\c&d\end{smallmatrix}\right)\in\Gamma_0(N)$ with $d\equiv n \mod N$. Define the $\chi$ eigenspace $$\mathcal{M}_k(N,\chi):=\{f\in \mathcal{M}_k(\Gamma_1(N)) |\langle n\rangle f=\chi(n)f ,\forall  n\in (\ZZ/N\ZZ)^\times\}.$$ Then one has the following  decomposion:

$$\mathcal{M}_k(\Gamma_1(N))=\bigoplus_{\chi}\mathcal{M}_k(N,\chi),$$ 
where $\chi$ runs through all Dirichlet characters mod $N$ such that $\chi(-1)=(-1)^k$. See Section 5.2 of \cite{diamond2006first} for details.
 
 Let $f\in \mathcal{M}_{k}(\Gamma_1(N))$, since $\left(\begin{smallmatrix}
 1&1\\
 0&1
 \end{smallmatrix}\right)\in \Gamma_1(N)$, one has a Fourier expansion: 
$$f(z)=\sum_{n=0}^{\infty}a_n(f)q^n, \quad q=e^{2\pi i z}.$$
If $a_0=0$ in the Fourier expansion of $f|_k[\alpha]$ for all $\alpha\in SL_2(\ZZ)$, then $f$ is called a cusp form. For a modular form $f\in \mathcal{M}_k(\Gamma_1(N))$ the Petersson inner product of $f$ with $g\in \mathcal{S}_k(\Gamma_1(N))$ is defined by the formula

$$\langle f,g\rangle =\int_{\mathcal{H}/\Gamma_1(N)}\overline{f(z)}g(z)y^{k-2}dxdy,$$
where $z=x+iy$ and $\mathcal{H}/\Gamma_1(N)$ is the quotient space. See \cite{diamond2006first} Section 5.4.  

We define the space $\mathcal{E}_k(\Gamma_1(N))$ of weight $k$ Eisenstein series with respect to $\Gamma_1(N)$ as the orthogonal complement of $\mathcal{S}_k(\Gamma_1(N))$ with respect to the Petersson inner product, i.e., one has the following orthogonal decomposition:

$$\mathcal{M}_k(\Gamma_1(N))=\mathcal{S}_k(\Gamma_1(N))\oplus \mathcal{E}_k(\Gamma_1(N)).$$

We also have decomposition of the cusp space and the space of Eisenstein series in terms of $\chi$ eigenspaces:

$$\mathcal{S}_k(\Gamma_1(N))=\bigoplus_{\chi}\mathcal{S}_k(N,\chi),$$

$$\mathcal{E}_k(\Gamma_1(N))=\bigoplus_{\chi}\mathcal{E}_k(N,\chi).$$  
See Section 5.11 of \cite{diamond2006first}.

For two Dirichlet characters $\psi$ modulo $u$ and $\varphi$ modulo $v$ such that $uv|N$ and $(\psi\varphi)(-1)=(-1)^k$ define 
\begin{equation}\label{eisen}
E_k^{\psi,\varphi}(z)=\delta(\psi)L(1-k,\varphi)+2\sum_{n=1}^{\infty}\sigma_{k-1}^{\psi,\varphi}(n)q^n,
\end{equation}

where $\delta(\psi)$ is 1 if $\psi=\mathbb{1}_1$ and 0 otherwise, $$L(1-k,\varphi)=-\frac{B_{k,\varphi}}{k}$$ is the special value of the Dirichlet $L$ function at $1-k$,
$B_{k,\varphi}$ is the $k$th generalized Bernoulli number defined by the equality

$$\sum_{k\geq 0}B_{k,\varphi}\frac{t^k}{k!}:=\sum_{a=0}^{N-1}\frac{\varphi(a)te^{at}}{e^{Nt}-1},$$

 and the generalized power sum in the Fourier coefficient is 

$$\sigma_{k-1}^{\psi,\varphi}(n)=\sum_{\substack{m|n\\ m>0}}\psi(n/m)\varphi(m)m^{k-1}.$$

See page 129 in \cite{diamond2006first}.

Let $A_{N,k}$ be the set of triples of $(\psi,\varphi,t)$ such that $\psi$ and $\varphi$ are primitive Dirichlet characters modulo $u$ and $v$ with $(\psi\varphi)(-1)=(-1)^k$ and $t\in \ZZ_{+}$ such that $tuv|N$. For any triple $(\psi,\varphi,t)\in A_{N,k}$, define 
\begin{equation}\label{eisen2}
E_{k}^{\psi,\varphi,t}(z)=E_{k}^{\psi,\varphi}(tz).
\end{equation}
Let $N\in \ZZ_{+}$ and let $k\geq 3$. The set

$$\{E_{k}^{\psi,\varphi,t}:(\psi,\varphi,t)\in A_{N,k}\}$$ 

gives a basis for $\mathcal{E}_k(\Gamma_1(N))$ and the set 

$$\{E_{k}^{\psi,\varphi,t}:(\psi,\varphi,t)\in A_{N,k},\psi\varphi=\chi\}$$ gives a basis for $\mathcal{E}_{k}(N,\chi)$ (see Theorem 4.5.2 in \cite{diamond2006first}).

For $k=2$ and $k=1$ an explicit basis can also be obtained but is more technical. Therefore we assume $k\geq 3$ when dealing with Eisenstein series.  For details about the remaining two cases see Chapter 4 in \cite{diamond2006first}.

Now we define Hecke operators. For a reference see page 305 in \cite{MR2153714} . 
Define $\Delta_m(N)$ to be the set

$$\left\{\gamma=\begin{pmatrix}
a&b\\
c&d\end{pmatrix}\in M_2(\ZZ)| c\equiv 0 \mod N,\det\gamma=m\right\}.$$
Then $\Delta_m(N)$ is invariant under right multiplication by elements of $\Gamma_0(N)$. One checks that the set
$$\mathscr{S}=\left\{\begin{pmatrix}
a&b\\0&d\end{pmatrix}\in M_2(\ZZ)| a,d>0,ad=m, b=0,\dots,d-1\right\}$$
forms a complete system of representatives for the action of $\Gamma_0(N)$. One defines the Hecke operator $T_m\in \End(\mathcal{M}_k(N,\chi))$ as:

$$f\mapsto f|_k T_m=m^{k/2-1}\sum_{\sigma}\chi(a_{\sigma})f|_k\sigma,$$

where $\sigma=\left(\begin{smallmatrix}a_\sigma & b_\sigma\\c_\sigma & d_\sigma\end{smallmatrix}\right)\in \mathscr{S}$ and $\gcd(m,N)=1$.

One can compute the action of $T_m$ explicitly in terms of the Fourier expansion:

$$f|_k T_m =a_0\sum_{m_1|m}\chi(m_1)m_1^{k-1}+\sum_{n=1}^{\infty}\sum_{m_1|(m,n)}\chi(m_1)m_1^{k-1}a_{mn/m_1^2}q^n.$$

The multiplication rule for weight $k$ operators $T_m$ is as follows:

\begin{equation}\label{multi}
T_mT_n=\sum_{m_1|(m,n)}\chi(m_1)m_1^{k-1}T_{mn/m_1^2}. 
\end{equation}

An eigenform $f\in　\mathcal{M}_k(N,\chi)$ is defined as the simultaneous eigenfunction of all $T_m$ with $\gcd(m,N)=1$.    

 If $f\in　\mathcal{M}_k(N,\chi)$ is an eigenform for the Hecke operators with character $\chi$, i.e., if 
\begin{equation}\label{eigenf}
f|_k T_m =\lambda_f(m)f \quad (\gcd(m,N)=1),
\end{equation}

 then equation (\ref{multi}) implies that 
 
 $$\lambda_f(m)\lambda_f(n)=\sum_{m_1|(m,n)}\chi(m_1)m_1^{k-1}\lambda_f(mn/m_1^2).$$

Comparing the Fourier coefficients in equation (\ref{eigenf}), one sees that
\begin{equation}\label{FC}
a_0\sum_{m_1|m}\chi(m_1)m_1^{k-1}=\lambda_f(m)a_0,
\end{equation}

$$\sum_{m_1|(m,n)}\chi(m_1)m_1^{k-1}a_{mn/m_1^2}=\lambda_f(m)a_n.$$
For $n=1$ one then has 

\begin{equation}\label{normalize}
a_m=\lambda_f(m)a_1.
\end{equation}

Therefore if $a_1\neq 0$ and we normalize $f$ such that $a_1=1$, then $a_m=\lambda_m(f)$.

Eisenstein series are eigenforms. Let $E_{k}^{\psi,\varphi,t}$ be the Eisenstein series defined by equation  (\ref{eisen2}). Then 
$$T_p E_k^{\psi,\varphi,t}=(\psi(p)+\varphi(p)p^{k-1})E_k^{\psi,\varphi,t}$$

if  $uv=N$ or $p\nmid N$. See Proposition 5.2.3 in \cite{diamond2006first}.

Let $M|N$ and $d|N/M$. Let $\alpha_d=\left(\begin{smallmatrix}
d&0\\
0&1
\end{smallmatrix}\right)$. Then  for any $f: \mathcal{H}\rightarrow\CC$ 
$$(f|_k[\alpha_d])(z)=d^{k-1}f(dz)$$
defines a linear map $[\alpha_d]$ taking $\mathcal{S}_{k}(\Gamma_1(M))$ to $\mathcal{S}_{k}(\Gamma_1(N))$.   

\vspace{0.1in}
\begin{definition}
Let $d$ be a divisor of $N$ and let $i_d$ be the map 
$$i_d:\quad (\mathcal{S}_{k}(\Gamma_1(Nd^{-1})))^2\rightarrow \mathcal{S}_k(\Gamma_1(N))$$ given by 
$$(f,g)\mapsto f+g|_k[\alpha_d].$$
The subspace of oldforms of level $N$ is defined by $$\mathcal{S}_k(\Gamma_1(N)):= \sum_{{p|N} \atop{p \text{ prime}}}i_p((\mathcal{S}_k(\Gamma_1(Np^{-1})))^2)$$

and the subspace of newforms at level $N$ is the orthogonal complement with respect to the Petersson inner product,

$$\mathcal{S}_k(\Gamma_1(N))^{\text{new}}=(\mathcal{S}_k(\Gamma_1(N))^{\text{old}})^{\bot}.$$
\end{definition}  

See Section 5.6 in \cite{diamond2006first}.

Let $\iota_d$ be the map $(\iota_df)(z)=f(dz)$. The main lemma in the theory of newforms due to Atkin and Lehner \cite{MR0268123} is the following:

\vspace{0.1in}
\begin{theorem}\label{main lemma} (Thm. 5.7.1 in \cite{diamond2006first})
If $f\in S_k(\Gamma_1(N))$ has Fourier expansion $f(z)=\sum a_n(f)q^n$ with $a_n(f)=0$ whenever $\gcd(n,N)=1$, then $f$ takes the form $f=\sum_{p|N}\iota_pf_p$ with each $f_p\in\mathcal{S}_k(\Gamma_1(N/p))$.
\end{theorem}

\vspace{0.1in}
\begin{definition} A newform is a normalized eigenform in $\mathcal{S}_k(\Gamma_1(N))^{\text{new}}$.
\end{definition}

By Theorem \ref{main lemma}, if $f\in \mathcal{M}_k(N,\chi)$ is a Hecke eigenform with $a_1(f)=0$, then $f$ is not a newform.
\section{ Identities between Hecke eigenforms}
 
In this section we study solutions to equation (\ref{polyid}) with $f,g,h$ being newforms with respect to $\Gamma_1(N)$. 
The first observation is that  if $a_f(0)=a_g(0)=0$ then $a_h(1)=0$, which implies $h=0$ by equation (\ref{normalize}), hence without loss of generality we can assume that $a_g(0)\neq 0$.  We normalize $g$ such that $a_g(0)=1$. In the following we talk about two cases according to whether or not $a_f(0)=0$. 
\vspace{0.1in}
\begin{lemma}\label{samech}
Assume that $g,f,h$ satisfies equation (\ref{polyid}) with $f$ and $g$ linearly independent over $\CC$, and let $\chi_g,\chi_f,\chi_h$ be the Dirichlet characters associated with $g,f,h$ respectively. Then $\chi_h=\chi_g^2=\chi_f^2=\chi_f\chi_g$.    
\end{lemma}
\begin{proof}
Take the diamond operator $\langle d\rangle$ and act on the identity $h=af^2+bfg+g^2$ to obtain $$\chi_h(d)h=a\chi^2_f(d)f^2+b\chi_f(d)\chi_g(d)fg+\chi^2_g(d)g^2.$$ 

We then substitute (\ref{polyid}) for $h$ to obtain
 \begin{equation}\label{Nonconstant}
 a(\chi_h(d)-\chi_f^2(d))f^2+b(\chi_h(d)-\chi_g(d)\chi_f(d))fg+(\chi_h(d)-\chi^2_g(d))g^2=0.
 \end{equation}
  
  Note that since $f$ and $g$ are linearly independent over $\CC$, so are $f^2,fg$ and $g^2$, which one can easily prove by considering the Wronksian. Then equation (\ref{Nonconstant}) implies that 
   $$a(\chi_h(d)-\chi_f^2(d))=b(\chi_h(d)-\chi_g(d)\chi_f(d))=\chi_h(d)-\chi^2_g(d)=0.$$
   Since $a,b\neq 0$, one finds
   $$\chi_h(d)=\chi_f^2(d)=\chi_f(d)\chi_g(d)=\chi_g^2(d).$$ Since $d$ is arbitrary we are done.
\end{proof}

\begin{proposition}\label{SameChar}
Suppose that the triple of newforms $(f,g,h)$ is a solution to (\ref{polyid}), $a_f(0)\neq 0$ and $a_g(0)\neq 0$. Then  $f=g$. 

\end{proposition}
\begin{proof}
By Lemma \ref{samech}, $f$ and $g$ have the same character. Denote it by $\chi$. Note that the space of weight $k$ newforms with character $\chi$ with nonvanishing constant term is of dimension 1 with basis $E_k^{\triv_1,\chi}$ by  equation (\ref{FC}). Since $f$ and $g$ are normalized, we have $f=g$.
\end{proof}

\begin{remark}
If $f=g$, then  equation (\ref{polyid}) reduces to $h=fg$, which was classified   in\cite{johnson2013hecke}. 
\end{remark}
Now we focus on the case where $a_f(0)=0$. We have the following solutions forced by dimension considerations. 

\vspace{0.1in}
\begin{lemma}
Let $f,g\in \mathcal{M}_k(N,\chi)$ be two eigenforms that are algebraically independent, and assume that $\dim_\CC \mathcal{M}_{2k}(N,\chi^2)\leq 3$. Then  every eigenform $h\in  \mathcal{M}_{2k}(N,\chi^2)$ satisfies an identity of the form $h=af^2+bfg+cg^2$.

\end{lemma}

Recall that the dimension of $\mathcal{M}_k(\Gamma)$ can be computed by the Riemann-Roch Theorem \cite{shimura1971arithmetic}. If $\Gamma=SL_2(\ZZ)$, it reads:

$$\dim_{\CC}\mathcal{M}_{k}=\begin{cases}
\left \lfloor{\frac{k}{12}}\right \rfloor &\mbox{if } k\equiv 2 \mod 12  \\
\left \lfloor{\frac{k}{12}}\right \rfloor +1 &\mbox{otherwise  }
\end{cases}$$

Hence for $k=12$ or $16$  one obtains the following examples computed by Sage \cite{sage}: 

$$E_{24}=a\Delta^2+bE_{12}\Delta+E_{12}^2,$$
         $$\displaystyle a=-\frac{2^{14}\cdot 3^8\cdot 5^4\cdot7^4\cdot13^2\cdot1571}{103\cdot691^2\cdot2294797},\quad \quad \displaystyle b=-\frac{2^8\cdot3^5\cdot5^3\cdot7^2\cdot13^3\cdot37}{103\cdot691\cdot2294797},$$

$$E_{32}=a(E_4 \Delta)^2+b(E_4\Delta) E_{16}+E_{16}^2,$$
 $$ a=\frac{-2^{18}\cdot 3^8\cdot5^5\cdot7^4\cdot11\cdot13^2\cdot17^2\cdot4273}{37\cdot683\cdot3617^2\cdot305065927}, $$ $$  b=\frac{-2^{12}\cdot3^4\cdot5^3\cdot7^2\cdot13\cdot17^2\cdot23\cdot1433}{37\cdot683\cdot3617\cdot305065927},$$

 where $\Delta=\prod_{n=1}^{\infty}q(1-q^n)^{24}$ is the cusp form of weight 12 for $SL_2(\ZZ)$ and 
 
 $$E_k=1-\frac{2k}{B_k}\sum_{n=1}^{\infty}\sigma_{k-1}(n)q^n$$
  is the Eisenstein series of weight $k$ for $SL_2(\ZZ)$.

The following two results follows from earlier work by Bruinier, Kohnen and Ono \cite{bruinier2004arithmetic}, which describes remarkable algebraic information contained in the zeros of Hecke eigenforms. Our independent proof starts from the point of view of polynomial identities.

Let $$A:=\{\rho\in \mathfrak{F}|E_{12k}(\rho)=0 \text{ for some }  k \in\ZZ_{+} \}$$ be the set of zeros  of Eisenstein series of weight $12k$ that are contained in the fundamental domain  $\mathfrak{F}$ of $SL_2(\ZZ)$. By work of
Rankin and Swinnerton-Dyer \cite{rankin1970zeros} we know that 
$A\subset\{e^{i\theta}|\frac{\pi}{3}\leq\theta\leq \frac{\pi}{2} \} $.

\vspace{0.1in}
\begin{theorem}\label{alge1}
For all $\rho\in A$, $E_{12}(\rho)/\Delta(\rho)$ is an algebraic number.
\end{theorem}

\begin{proof}
Note that the $n+1$  monomials $E_{12}(z)^{n-l}\Delta(z)^{l}$ are linearly independent over $\CC$ by considering the order of vanishing at infinity. One also has $\dim_{\CC}M_{12n}=n+1$, so the above monomials form a basis for $M_{12n}$. Therefore there exist $a_0,a_1,\cdots, a_n$ such that

\begin{equation}\label{algeq}
 E_{12n}(z)=\sum_{l=0}^{n}a_lE_{12}(z)^{n-l}\Delta(z)^{l}.
\end{equation}

For $\rho\in A $ we have $E_{12n}(\rho)=0$. Divide equation \ref{algeq} by  $\Delta(\rho)\neq 0$ to obtain

\begin{equation}
\sum_{l=0}^{n}a_l\left(E_{12}(\rho)/\Delta(\rho)\right)^{n-l}=0.
\end{equation}
To show that $E_{12}(\rho)/\Delta(\rho)$ is algebraic, it suffices to show that all the $a_l$ are rational. One sees this by the following algorithm to compute $a_l$:

First $a_0=1$ by considering the constant term.    

Suppose $a_0,\cdots,a_l$ are all rational. Then the Fourier coefficients of the function  $E_{12n}(z)-\sum_{k=0}^{l}a_kE_{12}(z)^{n-k}\Delta(z)^k$ are all rational since each term has rational Fourier coefficients. Further we know that the order of the function $E_{12n}(z)-\sum_{k=0}^{l}a_kE_{12}(z)^{n-k}\Delta(z)^k$  at $i\infty$ is $O(q^{l+1})$. 
Now set $$a_{l+1}=\frac{E_{12n}(z)-\sum_{k=0}^{l}a_kE_{12}(z)^{n-k}\Delta(z)^k}{\Delta^{l+1}}\bigg|_{i\infty}$$ as the constant term. Then one sees that $a_{l+1}$ is rational since the Fourier coefficients of the function $$\frac{E_{12n}(z)-\sum_{k=0}^{l}a_kE_{12}(z)^{n-k}\Delta(z)^k}{\Delta^{l+1}}$$ are all rational.
\end{proof}
\begin{corollary}\label{alge2}
For all $\rho\in A$, the function value of the discriminant $j(\rho)$ is algebraic. 
\end{corollary}
\begin{proof}
By the following well known identities:
$$\Delta(z)=\frac{E_4^3-E_6^2}{1728}$$
$$j(z)=\frac{E_4^3}{\Delta(z)}$$
$$E_{12}(z)=\frac{441E_4^3+250E_6^2}{691}$$
one then computes 
$$\frac{E_{12}(z)}{\Delta(z)}-j(z)=-\frac{432000}{691}.$$
Therefore $\play j(\rho)=\frac{E_{12}(\rho)}{\Delta(\rho)}+\frac{432000}{691}$ is an algebraic number by Theorem \ref{alge1}.
\end{proof}
Now we prove the finiteness result stated in the introduction. 
\vspace{0.1in}
\begin{theorem} 
	For all $N\in\ZZ_+$ and  $a,b\in\mathbb{C}^{\times}$, there are at most finitely triples $(f,g,h)$ of newforms with respect to $\Gamma_1(N)$ of weights $k>2$ satisfying the equation  $h=af^2+bfg+g^2$.
\end{theorem}

\begin{proof}
	By Proposition \ref{SameChar} and the remark thereafter
we only need to consider the case where $a_f(0)=0$ and $g$ is an Eisenstein series . Let $g=E_{k}^{\triv,\varphi}$,
$f\in \mathcal{M}_{k}(N,\varphi)$, $h= E_{2k}^{\triv,{\varphi}^2}$. For a fixed $N$ and $k$, the number of triples of such newforms $(f,g,h)$ is finite. We first show that $k$ is bounded.
Consider the Fourier expansions 
$$f=q+O(q^2)$$
$$g=1+\beta\left(\sum_{n=1}^{\infty}\sigma_{k-1}^{\varphi}(n)q^n\right)$$
$$h=1+\alpha\left(\sum_{n=1}^{\infty}\sigma_{2k-1}^{\varphi^2}(n)q^n\right)$$
where 
\begin{equation}\label{a}
\alpha=-\frac{4k}{B_{2k,\varphi^2}},
\end{equation}
 \begin{equation}\label{b}
 \beta=-\frac{2k}{B_{k,\varphi}},
 \end{equation}
$$\sigma_{k-1}^{\varphi}(n)=\sum_{d|n}\varphi(d)d^{k-1}.$$
Comparing the coefficients of $q$ on both sides of the equation $h=af^2+bfg+g^2$, one obtains
\begin{equation}\label{q}
b+2\beta=\alpha.
\end{equation}
Recall that we have  the following bound for the generalized Bernoulli number $B_{k,\chi}$, where $\chi$ is a primitive character with conductor $l_{\chi}$ and $\chi(-1)=(-1)^k$:
\begin{equation}\label{c}
2\zeta(2k)\zeta(k)^{-1}k!(2\pi)^{-k}l_{\chi}^{k-\frac{1}{2}}\leq |B_{k,\chi}|\leq 2\zeta(k)k!(2\pi)^{-k}l_{\chi}^{k-\frac{1}{2}}.
\end{equation}
See  \cite{johnson2013hecke}. Applying Stirling's formula in equation(\ref{c}) and using the fact that $\zeta(2k)/\zeta(k)$ is bounded  one sees that $|B_{k,\chi}|$ increases quickly as $k\rt\infty$. Then by equation (\ref{a}) and (\ref{b}) one sees that $\alpha,\beta\rt 0$ as $k\rt\infty$. The convergence is uniform for all $l_{\chi}$, hence for a fixed $b\neq 0$, $k$ is bounded for all $N$. 

For each fixed $k$, if $l_{\chi}\rt \infty$, then $|B_{k,\chi}|\rt \infty$ by equation (\ref{c}). Then equation (\ref{a})-(\ref{q}) show that $l_{\chi}$ is bounded. Note that $g,h$ are newforms, $\varphi$ and $\varphi^2$ are necessarily primitive, i.e., $l_\varphi=l_{\varphi^2}=N$. This shows that $N$ is bounded for each fixed $k$. Hence there can only be finitely many pairs $(N,k)$. Since for each pair there are only finitely many tuples of newforms, this shows the finiteness as claimed. 
 \end{proof}
%


\section{Maeda's conjecture and its applications }
In this section, following J.B. Conrey  and D.W. Farmer\cite{Farmer}, we will show how Maeda's conjecure provides information about Hecke eigenform identities. Throughout this section, all eigenforms are eigenforms for $SL_2(\ZZ)$ and we will write $\mathcal{S}_k$ instead of $\mathcal{S}_k(SL_2(\ZZ))$.

Recall that Maeda's conjecture (see \cite{Maeda} Conjecture 1.2) says the following:
\begin{conjecture}
	The Hecke algebra over $\QQ$ of $\mathcal{S}_k$ is
	simple (that is, a single number field) whose Galois closure over $\QQ$ has Galois
	group isomorphic to the symmetric group $\mathfrak{S}_n$, where $n = \dim \mathcal{S}_k$.
\end{conjecture}
The main result of this section is the following:
\vspace{0.1in}
\begin{theorem}\label{main}
	Assuming Maeda's conjecture for $\mathcal{S}_k$ and $\mathcal{S}_{2k}$, if $f\in \mathcal{S}_k$, $g\in \mathcal{S}_{2k}$ are any nonzero cusp eigenforms, then the Petersson inner product $\langle f^2,g\rangle $
	is nonzero.
\end{theorem}
An immediate corollary of Theorem \ref{main} is:
\vspace{.1in}
\begin{theorem}	
	Assuming Maeda's conjecture for all $\mathcal{S}_k$, then identities between cusp eigenforms for $SL_2(\ZZ)$ of the form $f^2=\sum_{i=1}^{d}c_ig_i$ are all forced by dimension considerations, where $f\in \mathcal{S}_{k}$, $g_i\in \mathcal{S}_{2k}$ and $c_i\ne 0$. In particular, there are only two identities  of the form $f^2=a_1g_1+a_2g_2$ given by Table 1.
	\end{theorem}
\begin{proof}
By Theorem \ref{main}, we know that $\langle f^2,g_i\rangle\ne 0$ for all $i=1,\cdots,\dim\mathcal{S}_{2k}$. Then we have $d=\dim \mathcal{S}_{2k}$, i.e., the given identity is forced by dimension considerations.
In particular, if $f^2=a_1g_1+a_2g_2$,  then $a_1\ne 0$ and $a_2\ne 0$ since $f^2$ is not itself an eigenform. Then, by Maeda's conjecture we have $\dim_{\CC}\mathcal{S}_{2k}=2$, and so $k=12$ or $16$. See Table 1 for the data on the two cases, which we computed with Sage\cite{sage}.
\end{proof}
\begin{table}
\begin{center}
 	\begin{tabular}{|c|c|c|c|}
 		\hline 
 		$f$ & $g_1$ & $g_2$ & $a_1=-a_2$ \\ 
 		\hline 
 		$\Delta$ & $E_{12}\Delta+(12\sqrt{144169}+\frac{32404}{691})\Delta^2$ & $\sigma(g_1)$ & $\frac{24}{\sqrt{144169}}$ \\ 
 		\hline 
 		$E_4\Delta$ & $\Delta(xE_4^5+(1-x)E_4^2E_6^2)$ & $\sigma(g_1)$ & $\frac{24}{\sqrt{18295489}}$\\ 
 		\hline 
 		\multicolumn{4}{|c|}{$x =\frac{12\sqrt{18295489}+20532}{1728}$ ,  $\sigma$ is the nontrivial element of $\Gal(F/\QQ)$}\\
 		\hline
 	\end{tabular}
 	\caption*{Table 1}
\end{center}
  \end{table}
We set up notation following \cite{Farmer} before giving the proof of Theorem \ref{main}. 
 
Recall that $\Delta = q\prod_{n=1}^{\infty}(1-q^n)$ is the discriminant and $$E_4=1+240\sum_{n=1}^{\infty}\sigma_3(n)q^n$$ is the Eisenstein series of weight $4$. Then the set 
 \begin{equation*} B_{\QQ}:=\{\Delta^jE_4^{\frac{k}{4}-3j}\} 
 \end{equation*} 
 forms a basis of $\mathcal{S}_k$ with integer Fourier coefficients\cite{Farmer}.  The matrix representation of the Hecke operator $T_n$ in the basis $B_{\QQ}$ has integer entries, hence the characteristic polynomial $T_{n,k}(x)$ of $T_n$ acting on $\mathcal{S}_k$ has integer coefficients and the eigenvalues $a_n$ of $T_n$ are all algebraic integers\cite{Farmer}. Let $d=\dim \mathcal{S}_k$. One defines the Hecke field $F$ associated with $\mathcal{S}_k$ by 
 $$F:=\QQ(a_j(n): 1\leq j\leq d, n\in\ZZ_{+}).$$
 
 The following two lemmas proved in \cite{Farmer} play an important role in this section.
 
 \vspace{.1in}
 \begin{lemma}
 	The Hecke field $F$ equals $\QQ(a_j(n):1\leq j,n\leq d)$. In particular, $F$ is a finite Galois extension of $\QQ$. 
 \end{lemma} 
 
 The Galois group $\Gal(F/\QQ)$ acts on functions with Fourier coefficients in $F$ in the following way:
 \begin{equation}
 \sigma\left(\sum a_n q^n\right)=\sum\sigma(a_n)q^n.
 \end{equation} 
 
 \begin{lemma}
 	The group $\Gal(F/\QQ)$ acts on the set of $B =\{f_1,\cdots,f_d\}$ of normalized cusp eigenforms. If $T_{n,k}(x)$ is irreducible for some $n$, then the action is transitive. Furthermore, if $T_{n,k}(x)$ is irreducible then $F/\QQ$ is the splitting field of $T_{n,k}(x)$ and $\Gal(F/\QQ) =\Gal(T_{n,k})$.
 \end{lemma}
 
 Now we begin the proof of Theorem \ref{main}.

 \begin{proof}
  Let $f\in \mathcal{S}_{k}$ be a normalized Hecke eigenform. Since $f^2\in \mathcal{S}_{2k}$ we expand $f^2$ in a normalized eigenbasis to obtain 
 \begin{equation} \label{identity}
 f^2=\sum_{i=1}^{d_2}c_ig_i
 \end{equation}
 for some $c_i\in\CC$ and $d_2=\dim \mathcal{S}_{2k}$. 
  Let $F_1$ (resp. $F_2$) be the Hecke field for $\mathcal{S}_k$ (resp. $\mathcal{S}_{2k}$) and $F_1F_2$ be the composite field of $F_1$ and $F_2$. Then we have the following: 
 \begin{lemma}\label{composite}
 	$c_i\in F_1F_2$ for  $1\leq i\leq d_2$.
 	
 \end{lemma}
 
 \begin{proof} 
 	
 	 Comparing the Fourier coefficients in equation (\ref{identity}), one obtains the linear equations 
 	\begin{equation}
 	\sum_{i=1}^{d_2}c_ia_n(g_i)=a_n(f^2)
 	\end{equation}
 	for $1\leq n\leq d_2$. We claim that the coefficient matrix $(a_n(g_i))_{{1\leq n,i\leq d_2}}$ is nonsingular. Assume our claim for now.  Solve $c_i$ by Cramer's rule and note that $a_n(g_i)\in F_2$ and $a_n(f^2)\in F_1$, then one sees $c_i\in F_1F_2$. 
 	
 	Now we prove our claim. Suppose for a contradiction that the coefficient matrix $(a_n(g_i))_{{1\leq n,i\leq d_2}}$ is singular, then there exit $c_1,\cdots, c_n\in\CC$ such that $$0=\sum_{i}c_ia_n(g_i)=a_n(\sum_{i}c_ig_i)$$ for $1\leq n\leq d_2$, i.e., the cusp form $g:=\sum_{i}c_ig_i\in \mathcal{S}_{2k}$ satisfies $a_n(g)=0$ for $1\leq n\leq d_2$. The order of vanishing of $g$ at $\infty$ satisfies $v_{\infty}(g)\geq d_2+1=\lfloor\frac{2k}{12}\rfloor+1$. Recall that the valence formula (Theorem VII. 3.(iii) in \cite{MR0344216}) for $g$ reads
$$v_{\infty}(g)+\frac{1}{2}v_{i}(g)+\frac{1}{3}v_{\omega}(g)+\sum_{\substack{\tau\in\mathcal{H}/SL_2(\ZZ)\\ \tau\neq i,\omega}}v_{\tau}(g)=\frac{2k}{12}.$$  

Since all the terms on the left are positive, one obtains the contradiction $\lfloor\frac{2k}{12}\rfloor+1\leq \frac{2k}{12} $. This proves our claim, and hence the lemma.
 \end{proof}
 
 \begin{lemma}\label{tran2} 
 	If the stabilizer of $f^2$ in $\Gal(F_1F_2/\QQ)$ acts transitively on the set $B_2=\{g_1,\cdots,g_{d_2}\}$ of cusp eigenforms of weight $2k$, then we have $\langle f^2,g_i\rangle \neq 0$ for $1\leq i\leq d_2$.
 \end{lemma} 
 \begin{proof}
 	Suppose by contradiction $\langle f^2,g_i\rangle=0$. For equation (\ref{identity}) we take the inner product with $g_i$ and use the orthogonality relation $\langle g_i,g_j\rangle =\delta_{ij}$ to find $c_i=0$. Then transitivity gives  $c_i=0$ for $1\leq i\leq d_2$, which implies that $f^2=0$, which is a contradiction. 
 \end{proof}
 
 \begin{lemma}\label{tran}
 	Assuming Maeda's conjecture for $\mathcal{S}_k$ and $\mathcal{S}_{2k}$, the stabilizer of $f^2$ in $\Gal(F_1F_2/\QQ)$ acts transitively on the set $B_2$ of cusp eigenforms of weight $2k$. 
 \end{lemma} 
 \begin{proof}
 	
 	We have 
 	$$
 	\Gal(F_1F_2/\QQ)\cong \{(\sigma_1,\sigma_2)\in \Gal(F_1/\QQ)\times \Gal(F_2/\QQ)\big| \sigma_1|_{F_1\cap F_2}=\sigma_2|_{F_1\cap F_2}\}.$$
  Note that $F_1\cap F_2$ is Galois over $\QQ$,since $F_1,F_2$ are both Galois over $\QQ$, and by elementry Galois theory, the subgroup $H_2$ of $\Gal(F_2/\QQ)$ that fixes  $F_1\cap F_2$ is normal.   
 	Assuming Maeda's conjecture, we have $\Gal(F_2/\QQ)\cong \mathfrak{S}_{d_2}$, $\Gal(F_1/\QQ)\cong \mathfrak{S}_{d_1}$, where $d_1=\dim \mathcal{S}_k$ and $d_2=\dim \mathcal{S}_{2k}$.  
 	
 	\medskip

 	If $d_2\geq 5$, then the only  normal subgroups of $\mathfrak{S}_{d_2}$ are $\mathfrak{S}_{d_2}$, the alternating group $A_{d_2}$ and $\{1\}$. Since $F_1\cap F_2\subsetneq F_2$ by degree considerations, we have $ H_2\neq \{1\}$, therefore $H_2\cong A_{d_2}$ or $H_2\cong \mathfrak{S}_{d_2}$. 
 	
 	\medskip
 
 	If $H_2\cong \mathfrak{S}_{d_2}$, then $F_1\cap F_2=\QQ$ and 
 	$$\Gal(F_1F_2/\QQ)\cong \Gal(F_1/\QQ)\times \Gal(F_2/\QQ)\cong \mathfrak{S}_{d_1}\times \mathfrak{S}_{d_2},$$  
 	using Maeda's conjecture to obtain the second isomorphism. The stabilizer of $f^2$  in $\Gal(F_1F_2/\QQ)$ contains a subgroup  isomorphic to $\mathfrak{S}_{d_2}$, hence acts transitively on $B_2$.

\medskip
 	
 	If $H_2\cong A_{d_2}$, then $$[F_1\cap F_2:\QQ]=|\Gal(F_1\cap F_2/\QQ)| =|\Gal(F_2/\QQ)|/|H_2|=2,$$ so $F_1\cap F_2$ is a quadratic field. Moreover we know that the subgroup of $\mathfrak{S}_{d_1}$ that fixes $F_1\cap F_2$ elementwise is a normal subgroup of index two, i.e., $A_{d_1}$. In this case, $F_1\cap F_2$ is generated by the square root of the discriminant of $T_{n,k}(x)$ or $T_{n,2k}(x)$.  
 	Then we have 	
 	$$\Gal(F_1F_2/\QQ)\cong\{(\sigma_1,\sigma_2)\in \mathfrak{S}_{d_1}\times \mathfrak{S}_{d_2}| \sgn \sigma_1=\sgn \sigma_2\}.$$
 	In this case, the stabilizer of $f^2$ in $\Gal(F_1F_2/\QQ)$ contains a subgroup isomorphic to $A_{d_{2}}$, which acts transitively on $B_2$.
 	
 	If $d_2\leq 4$, we have $\lfloor{2k/12}\rfloor\leq 4$, so $k\leq 28$. It suffices to show that for $12\leq k\leq 28$, $F_1\cap F_2=\QQ$. We only need to check the cases $(k,2k)=(24,48),(28,56)$ since for all other cases $\dim\mathcal{S}_k=1$ and $F_1=\QQ$. We check these two cases by showing that no prime ramifies in $F_1\cap F_2$. Then we obtain that $F_1\cap F_2=\QQ$ by Minkowski theorem. Indeed, if $p$ is a prime that ramifies in $F_1\cap F_2$, then $p$ ramifies in $F_1$ and $F_2$. Let $\alpha_1,\cdots,\alpha_{d_1}$ be the $d_1$ roots of the polynomial $T_{n,k}(x)$. Note that the splitting field $F_1$ of $T_{n,k}(x)$ can be written as the composite field of the $d_1$ isomorphic subfields $\QQ(\alpha_i)$, then we see that $p$ ramifies in each $\QQ(\alpha_i)$ by the tower property for ramification. Thus $p|D_{\QQ(\alpha_i)}$, where $D_{\QQ(\alpha_i)}$ is the discriminant of the field $\QQ(\alpha_i)$. Similarly $p|D_{\QQ(\beta_j)}$, where $\beta_1,\cdots,\beta_{d_2}$ are the $d_2$ roots of the polynomial $T_{n,2k}(x)$. Therefore $p|\gcd(D_{\QQ(\alpha_i)},D_{\QQ(\beta_j)})$. Thus it suffices to check that $\gcd(D_{\QQ(\alpha_i)},D_{\QQ(\beta_j)})=1$ for the two cases. One verifies this directly by Table 2, where the  first column is the weight and the second column is the discriminant $D_{\QQ(\alpha_i)}$ or $D_{\QQ(\beta_j)}$ computed by Sage\cite{sage}. 
 	 \end{proof} 
 	\begin{table}\label{table2}
 		\begin{center}
 		\begin{tabular}{ |c|l| }
 			\hline
 			$k$ & $D_{\QQ(\alpha_i)}$\\ \hline
 			 			24 & 144169 \\ \hline
 			
 			28 & 131 * 139 \\ \hline
 			
 			48 & 31 * 6093733 * 1675615524399270726046829566281283 \\ \hline
 			
 			56 & 41132621 * 48033296728783687292737439509259855449806941 \\ \hline
 		\end{tabular}
 		\caption*{Table 2}
 		\label{table:coprime}
 		\end{center}
 	\end{table}

 Now Theorem \ref{main} follows easily by combining Lemmas \ref{composite}-\ref{tran}.
\end{proof}
 \section*{Acknowledgements}
The author would like to thank Professor Matthew Stover for his guidance and encouragement on this research. Without his support, this paper could never come into being. The author also would like to thank Professor Benjamin Linowitz for his comments on an earlier draft of this paper.

\bibliographystyle{plain}
\bibliography{eigenformIdentities}

\end{document}